\documentclass[11pt]{amsart}
\usepackage{color, amssymb, amsmath,amsfonts,mathrsfs,tikz}
\usetikzlibrary{arrows.meta}
\usepackage{enumitem}

\usepackage{ mathrsfs, amsmath, amsthm, amssymb,
thm-restate, pifont}
\usepackage{textcomp}

\usepackage[colorlinks]{hyperref}
\hypersetup{linkcolor=blue, urlcolor=blue, citecolor=red}

\usepackage[capitalise]{cleveref}

\usepackage[colorlinks]{hyperref}
\hypersetup{linkcolor=blue, urlcolor=blue, citecolor=red}
\usepackage[capitalise]{cleveref}

\newtheorem{theorem}{Theorem}[section]

\newtheorem{prop}[theorem]{Proposition}

\newtheorem{lemma}[theorem]{Lemma}
\newtheorem{question}[theorem]{Question}
\newtheorem*{claim}{Claim}

\newtheorem{definition}[theorem]{Definition}

\newcommand{\N}{\mathbb N}

\newcommand{\Tau}{\mathcal T}

\newcommand{\Z}{\mathbb Z}
\newcommand{\im}{\operatorname{im}}
\newcommand{\dom}{\operatorname{dom}}

\newcommand{\Homeo}{\operatorname{Homeo}}

\newcommand{\makeset}[2]{\left\lbrace #1 \;\middle|\;
 \begin{tabular}{@{}l@{}}
   #2
  \end{tabular}
  \right\rbrace}
  \usepackage{url, amsmath, amsthm, amssymb}
\usepackage{amsmath}
\usepackage{multicol}
\usepackage{hyperref}
\usepackage{multicol}
\usepackage{mathrsfs}
\usepackage{abstract}
\hypersetup{
    colorlinks,
    citecolor=black,
    filecolor=black,
    linkcolor=black,
    urlcolor=black
}

\author{L Elliott}

\email{luna.elliott142857@gmail.com}
\makeatletter
\@namedef{subjclassname@2020}{%
  \textup{2020} Mathematics Subject Classification}
\makeatother
\title{The Zariski Topology on Homeomorphism groups}

\begin{document}
\maketitle

\begin{abstract}
    The Zariski topology $\mathfrak{Z}_G$ on a group $G$ is the coarsest topology such that all sets of the form $\{x \in G \mid 1_G \neq g_0 x^{k_0} g_1 \dots g_{l-1} x^{k_{l-1}} g_l\}$ are open. Originally introduced by Bryant as the ``verbal topology,'' it serves as a fundamental tool for investigating the topological structure of infinite groups and is always a \(T_1\) topology with continuous shifts and inversion.
    Since the Zariski topology is coarser than every Hausdorff group topology on $G$, it provides a natural starting point for topologizing groups; specifically, for countable or abelian groups, it is known that the Zariski topology coincides with the Markov topology—the intersection of all Hausdorff group topologies on $G$. 

    In this paper, we analyze the Zariski topology on various homeomorphism groups. We demonstrate that for the Thompson groups $F$ and $T$, the Zariski (and thus Markov) topology coincides with the standard compact-open topology derived from their respective actions on $[0,1]$ and $S^1$. In contrast, we show that the Zariski (and thus Markov) topology on Thompson's group $V$ is irreducible, and therefore neither Hausdorff nor a group topology. As $V$ acts highly transitively on each of its orbits, this result stands in notable opposition to a theorem by Banakh et al, which establishes that the Zariski topology on any permutation group containing all finitely supported permutations is a Hausdorff group topology.

    Our results for the Zariski topologies on \(F,T\) and \(V\) also apply to the full homeomorphism groups $\operatorname{Homeo}([0,1])$, $\operatorname{Homeo}(S^1)$ and $\operatorname{Homeo}(2^\omega)$ respectively. We conclude by providing a classification of the connected manifolds $M$ for which the homeomorphism group $\mathrm{Homeo}(M)$ admits a Hausdorff Zariski topology.
\end{abstract}
\tableofcontents
\section{Notice}

In the original draft of this paper we missed some important references \cite{GARTSIDE2003103,interval1,chang2017minimum}. This was kindly pointed out to the author by Dikranjan and results from these references are exposited in section 7 of \cite{zbMATH06893997}. In  \cite[Theorem 1]{GARTSIDE2003103} (see also \cite[Theorem 7.5]{zbMATH06893997}), Gartside and Glyn show that the compact-open topology is the minimum group topology on any metric one-dimensional manifold. This implies part of our Theorem 1.5.
The result \cite[Theorem 3.4]{interval1} (see also \cite[Theorem 7.6]{zbMATH06893997}) of Megrelishvili and Polev implies that the Zariski topology on the homeomorphism group of the interval \([0,1]\) is the compact-open topology. 
This theorem also applies to the order preserving subgroup and it is commented in \cite{interval1} that the proofs can be modified to work with the circle \(S^1\). These results imply our Theorems 1.1 and 1.2 in the cases of \(\operatorname{Homeo}([0,1]),\operatorname{Homeo}_+([0,1]),\operatorname{Homeo}(S^1)\) and \(\operatorname{Homeo}_+(S^1)\).
In \cite[Corollary 2]{chang2017minimum} (see also \cite[Theorem 7.7]{zbMATH06893997}), Chang and Gartside show that the intersection of the pointwise and compact-open topologies on the homeomorphism groups of any of:
\begin{itemize}
    \item a compact manifold of dimension at least 2,
    \item \(2^\omega\),
    \item \([0,1]^\omega\)
\end{itemize}
 contains no indiscrete group topology. 
 This implies that the Zariski topology on \(\operatorname{Homeo}(2^\omega)\) is not a group topology, which is part of our Theorem 1.3.  
The result \cite[Corollary 6]{chang2017minimum} shows that the compact-open topology on the homeomorphism
group of a compact manifold with non-trivial boundary is not a minimal Hausdorff group topology. Moreover, as commented below \cite[Theorem 7.9]{zbMATH06893997}, the proof of \cite[Theorem 7]{chang2017minimum} implies that the Zariski topology on the Homeomorphism groups of the inverval $[0,1]$ and the circle $S^1$ are the compact-open topologies. In the cases of the full homeomorphism groups, this again implies our Theorems 1.1 and 1.2.

\section{Introduction}

A \emph{group topology} on a group \((G,*, {}^{-1})\) is a topology on the underlying set \(G\) such that the operations \(*:G\times G\to G\) and \({}^{-1}:G\to G\) are continuous.
A \emph{topological group} is then a group equipped with a group topology.
    Having a topology on a group is often a very useful tool for analyzing its group structure, particularly if the topology is in some sense intrinsic.
    Topologies on homeomorphism groups in particular have seen special attention in the literature, see for example \cite{arens1946topologies,dieudonne1948topological,dijkstra2005homeomorphism,mostert1961reasonable}.
    The compact-open topology on the homeomorphism group of a space \(X\), generated by the sets
   \[\makeset{f\in \operatorname{Homeo}(X)}{\((K)f\subseteq U\)}\]
where \(K\subseteq X\) is compact and \(U\subseteq X\) is open, is perhaps the most standard.
    
    A fundamental tool for studying topologies on a group is the \emph{group Zariski topology} $\mathfrak{Z}_G$, introduced by Bryant \cite{bryant1977verbal} as the ``verbal topology.'' There is also a notion of the semigroup Zariski topology for semigroups which can potentially be coarser than the group Zariski topology when applied to a group, the definition is analogous but uses non-solutions to semigroup equations instead of group equations (which disallows inverses). A comparison of the two can be found in \cite{bardyla2025note}.
    In this paper the ``Zariski topology" on a group always refers to the group Zariski topology.
The Zariski topology on groups and semigroups has been extensively studied in the literature. Some example papers include \cite{bryant1977verbal,DDSabelianzariski,banakh2012algebraically,bardyla2025note,pinsker2023zariski,dikranjan2012markov,banakh2010zariski,goffer2024note,elliott2023automatic,marimon2025guide, elliott2022constructing, zbMATH06893997,zbMATH07781608,zbMATH06300237,zbMATH05232892}.
The Zariski topology on a group \(G\) has a subbase consisting of the non-solution sets to equations of the form \(1_G=g_0 x^{k_0} g_1 \dots g_{l-1} x^{k_{l-1}} g_l\) for \(k_0,\ldots, k_{l-1}\in \Z\) and \(g_0,\ldots, g_l\in G\). As the map \(x\mapsto g_0 x^{k_0} g_1 \dots g_{l-1} x^{k_{l-1}} g_l\) is continuous in any topological group and the equalizer of two continuous maps to a Hausdorff space is always closed,  the Zariski topology provides a lower bound for the Hausdorff group topologies on any group $G$. 
For countable groups or abelian groups, this significance is amplified by the work of \cite{markov1946unconditionally, DDSabelianzariski}, which implies that the Zariski topology coincides with the intersection of all Hausdorff group topologies on $G$. This intersection topology is called the \emph{Markov Topology}.
As such, studying the Zariski topology informs us about all Hausdorff group topologies and, in the case of a countable or abelian group, finding a set which is not Zariski-open amounts to finding a Hausdorff group topology on our group with respect to which the set is not open. It should be noted however that the Zariski topology on a group is not always a group topology.

In this paper, we investigate the Zariski topology on various homeomorphism groups. 
We start by focusing on the famous Thompson groups $F$, $T,$ and $V$ and then generalize to larger groups. The groups \(T\leq \operatorname{Homeo}(S^1)\) and \(V\leq \operatorname{Homeo}(2^\omega)\) in particular were the first known examples of finitely presented infinite simple groups.
See \cite{cannon1996introductory} for an introduction to \(F\), \(T\) and \(V\).

We demonstrate that for $F$ and $T$, the Zariski topology recovers the classical compact-open topology derived from their respective actions on $[0,1]$ and $S^1$. As all the groups in  \cref{mainF} and \cref{mainV} also satisfy the hypothesis of \cite{bardyla2025note} Theorem 1.12, this also provides a new infinite family of groups whose semigroup Zariski topology differs from the group Zariski topology.

\begin{theorem}\label{mainF}
    If \(G\leq \Homeo([0,1])\) is any group containing Thompson's group \(F\) (or any of the Thompson groups \(F_n\) with \(n\geq 2\)), then the group Zariski topology on \(G\) is the compact-open topology. In particular, it is a Hausdorff group topology.
\end{theorem}

\begin{theorem}\label{mainT}
    If \(G\leq \Homeo(S^1)\) is any group containing Thompson's group \(T\) (or any of the Thompson groups \(T_n\) with \(n\geq 2\)), then the group Zariski topology on \(G\) is the compact-open topology. In particular, it is a Hausdorff group topology.
\end{theorem}

This is perhaps not so surprising, as by a theorem of Rubin (see \cite{rubin1989reconstruction,belk2025short}) the actions of the groups \(F\), \(T\), \(V\) are algebraically recoverable.
However, perhaps more surprisingly, we find a very different outcome for $V$. 
We prove that the Zariski topology on $V$ is irreducible (so in particular not Hausdorff). 
It follows that the techniques used in Rubin's theorem are not suited to handling Zariski topologies.
As \(V\) is closely tied to symmetric groups, this result also contrasts with a result of  \cite{banakh2012algebraically} which says the Zariski topology on any permutation group containing the finite support elements is Hausdorff.

\begin{theorem}\label{mainV}
    The group Zariski topologies on Thompson's group  \(V\) (or any of the Thompson groups \(V_n\) with \(n\geq 2\)) and the group \(\operatorname{Homeo}(2^\omega)\) are irreducible. In particular, these are neither Hausdorff nor group topologies.
\end{theorem}
In proving \cref{mainV}, we develop the following more general theorem of independent interest.
\begin{theorem}\label{maingeneral}
    Suppose that \(X\) is a Hausdorff topological space with no isolated points and \(G\) acts highly transitively on \(X\) by homeomorphisms.
    In this case, the group Zariski topology on \(G\) is irreducible.
\end{theorem}

 We conclude by classifying the connected manifolds $M$ for which $\operatorname{Homeo}(M)$ supports a Hausdorff Zariski topology.
In this paper, by an \(n\)-manifold, we mean a second countable Hausdorff space such that every point has an open neighbourhood homeomorphic to \(\mathbb{R}^n\). In particular, we do not consider manifolds with boundary.
\begin{theorem}\label{mainManifold}
    Suppose that \(n\in \N\) and \(M\) is an \(n\)-manifold with exactly one connected component. Then the following are equivalent:
    \begin{enumerate}
        \item the group Zariski topology on \(\operatorname{Homeo}(M)\) is Hausdorff,
        \item the group Zariski topology on \(\operatorname{Homeo}(M)\) is a group topology,
        \item \(n\leq 1\).
    \end{enumerate}
\end{theorem}

\section{Preliminaries}
Here we give the standard definitions, conventions and facts used throughout the paper.
In this paper, we use right actions and compose functions from left to right. We also adopt the convention that \(0\in \N\).
We begin with the definition of the Zariski topology.
\begin{definition}[Zariski Topology]\label{zariskidef}
    Suppose that \(G\) is a group with identity \(1_G\). Then the group Zariski topology \(\mathfrak{Z}_G\) is the topology with subbasis consisting of the sets
    \[\{x \in G \mid 1_G \neq g_0 x^{k_0} g_1 \dots g_{l-1} x^{k_{l-1}} g_l\}\]
    where \(l\in \N\), \(g_0,\ldots,g_{l}\in G\) and \(k_0,\ldots, k_{l-1}\in \Z\).
\end{definition}

The following is perhaps the most standard topology to consider for a homeomorphism group. We will show that in many cases it is the smallest Hausdorff topology.

\begin{definition}[Compact-open topology]
Suppose that \(X\) is a topological space. Then the compact-open topology on \(\operatorname{Homeo}(X)\) is the topology with subbasis consisting of the sets
\[\makeset{f\in \operatorname{Homeo}(X)}{\((K)f\subseteq U\)}\]
where \(K\subseteq X\) is compact and \(U\subseteq X\) is open.
\end{definition}

This definition of the compact-open topology, while standard, is not always the easiest to use.
Consequently, we provide an alternative characterization.
The following fact is well-known 
(see for example Theorem 46.8 of \cite{munkres2000topology} and section 9B example 8 of \cite{kechris1995classical}).
\begin{prop}\label{compact-open-good}
If \((X,d)\) is a compact metric space then \(\operatorname{Homeo}(X)\) is a topological group with the compact-open topology. 
Moreover this topology is induced by the metric
\[d_\infty(f,g)=\sup_{x\in X} d((x)f,(x)g).\]
\end{prop}

Our results on non-Hausdorff topologies in this paper actually show a much stronger condition. This condition prevents any pairs of non-empty disjoint open sets from existing at all.
\begin{definition}[Irreducible]
    A topological space \(X\) is said to be \emph{irreducible} (also known as hyperconnected) if whenever \(K_1,K_2\subsetneq X\) are closed, we have \(K_1\cup K_2 \neq X\).
    Equivalently, all non-empty open subsets of \(X\) are dense.
\end{definition}

\begin{definition}[Support]
    If \(g\) is a permutation of a set \(X\), then the support of \(g\) is the set
    \(\makeset{x\in X}{\((x)g\neq x\)}.\)
\end{definition}
If \(f\) and \(g\) are permutations of a common set, then we denote by \(f^g\) the conjugate \(g^{-1}fg\) of \(f\) by \(g\).
Crucially, the support of \(f^g\) is always the image of the support of \(f\) under the map \(g\).
\begin{definition}[Highly Transitive]
    We say that a group \(G\) of permutations of a set \(X\) is \emph{highly transitive} if for every bijection \(\sigma\) between finite subsets of \(X\), there is \(g\in G\) extending \(\sigma\).
\end{definition}

The Thompson groups $F_n$, $T_n$ and $V_n$ are the homeomorphism groups defined as follows:
\begin{itemize}
    \item $F_n$ is the group of piecewise linear homeomorphisms of $[0,1]$ that are orientation-preserving, have all slopes in the set $\{n^k \mid k \in \mathbb{Z}\}$ and have breakpoints in the $n$-adic rationals $\mathbb{Z}[\frac{1}{n}]$.
    \item $T_n$ is the group of piecewise linear homeomorphisms of the circle $S^1 \cong \mathbb{R}/\mathbb{Z}$ satisfying the same slope and breakpoint conditions as $F_n$.
    \item $V_n$ is the group of homeomorphisms \(f\) of the Cantor space $\{0, 1, \dots, n-1\}^\omega$ that are piecewise defined by prefix replacements. That is to say, for all \(x=x_0x_1x_2\ldots\in \{0, 1, \dots, n-1\}^\omega\), there are \(i,j\in \N\) and \(y_0y_1\ldots y_j\in \{0, 1, \dots, n-1\}^j\) with \((x_0x_1x_2\ldots x_iz_0z_1\ldots)f=y_0y_1y_2\ldots y_jz_0z_1\ldots\) for all \(z_0z_1\ldots\in \{0, 1, \dots, n-1\}^\omega\).
\end{itemize}

Finally, in the manifolds section, we will need a few standard facts about manifolds. 
The first can be found in \cite{mann2021structure}, while the second is the primary focus of \cite{gale1987classification}.
\begin{prop}\label{hightrans}
    If \(n\in \N\backslash \{0,1\}\) and \(M\) is an \(n\)-manifold with exactly one connected component, then \(\operatorname{Homeo}(M)\) is highly transitive. 
\end{prop}
\begin{prop}\label{1manclass}
    If \(M\) is a \(1\)-manifold with exactly one connected component, then \(M\) is either homeomorphic to \( \mathbb{R}\) or \(S^1\cong \mathbb{R}/\mathbb{Z}\).
\end{prop}

\section{The Cantor Space}

In this section we prove the main results about the homeomorphism group of the Cantor space and Thompson's group $V$.
We start with a technical lemma giving a more restricted description of a basis for the Zariski topology on an arbitrary group.

\begin{lemma}\label{lem:techZariski}
    Suppose that \(G\) is a group.
    Let \(\mathcal{B}_G\) denote the family of sets of the form
    \begin{equation}\label{bigeq}
       \makeset{x\in G}{\(1_G\neq x^{k_{0,0}}a_{0,1}\ldots a_{0,l_0-1}x^{k_{0,l_0-1}}a_{0,l_0}\)\\
   \(1_G\neq x^{k_{1,0}}a_{1,1}\ldots a_{1,l_1-1}x^{k_{1,l_1-1}}a_{1,l_1}\)\\
   \quad\quad\vdots\\
   \(1_G\neq x^{k_{m-1,0}}a_{m-1,1}\ldots a_{m-1,l_{m-1}-1}x^{k_{m-1,l_{m-1}-1}}a_{m-1,l_{m-1}}\)}
    \end{equation}
    where  
    \begin{itemize}
    \item \(m\in \N\), \(l_0,l_1,\ldots, l_{m-1}\in \N\backslash \{0\}\),
        \item  \(a_{i,j}\in G\) for \(i<m\) and \(0<j\leq l_i\),
        \item  \(a_{i,j}\neq 1_G\) for \(i<m\) and \(0<j<l_i\),
        \item  \(k_{i,j}\neq 0\) for \(i<m\) and \(0\leq j<l_i\).
    \end{itemize}
    The set \(\mathcal{B}_G\) is a basis for the Zariski topology on \(G\).
    Moreover, if \(\varnothing\notin \mathcal{B}_G\), then the Zariski topology is irreducible.
\end{lemma}
\begin{proof}
  The Zariski topology \(\mathfrak{Z}_G\), by definition, has a subbase consisting of the sets
    \[\makeset{x\in G}{\(1_G\neq a_{0}x^{k_0}a_1\ldots a_{l-1}x^{k_{l-1}}a_l\)}\]
    where  \(a_0,a_1\ldots a_{l}\in G\), \(k_{0},k_{1},\ldots, k_{l-1}\in \Z \). 
    By applying conjugation, grouping terms and omitting trivial cases (the empty and universal sets), we may assume without loss of generality that the subbasic sets are defined such that:
    \begin{enumerate}
        \item  \(a_0=1_G\),
        \item  \(l>0\),
        \item \(k_i\neq 0\) for all \(i<l\),
        \item \(a_i\neq 1_G\) for \(0<i< l\).
    \end{enumerate}
    The family of finite intersections of the remaining sets gives us precisely the basis \(\mathcal{B}_G\).
    
    Now suppose that \(\varnothing\notin \mathcal{B}_G\). We show that the Zariski topology is irreducible. Let \(U\) and \(W\) be two non-empty Zariski open subsets of \(G\). We must show that \(U\cap W\) is non-empty.
    As \(\mathcal{B}_G\) is a basis, there exist \(B_U, B_W\in \mathcal{B}_G\) with \(B_U\subseteq U\) and \(B_W\subseteq W\).
    Note also that \(B_U\cap B_W\in \mathcal{B}_G\). Thus \(B_U\cap B_W\neq \varnothing\).
    It follows that
    \[U\cap W \supseteq B_U\cap B_W \neq\varnothing\]
    as required.
    
\end{proof}

The following lemma will be our main tool for showing that Zariski topologies are irreducible.
It closely resembles Theorem~1.12 of \cite{bardyla2025note}, which shows that the semigroup Zariski topology on a group is irreducible under weaker hypotheses. 
These weaker hypotheses are not sufficient for the group Zariski topology as can be seen by considering the group \(\operatorname{Sym}(\N)\) and Theorem~2.1 of \cite{banakh2012algebraically} or indeed Thompson's group \(T\) and \cref{mainT}. 
The stronger assumptions however are still satisfied for a large family of examples.

\begin{lemma}\label{main}
    Suppose that \(X\) is an infinite set and \(G\) is a highly transitive group of permutations of \(X\) such that every non-identity element has infinite support.
    In this case the Zariski topology on \(G\) is irreducible.
\end{lemma}
\begin{proof}
   Recall the basis \(\mathcal{B}_G\) from \cref{lem:techZariski}.
Let \(m\), \(l_i\), \(a_{i,j}\) and \(k_{i,j}\) and \(S\in \mathcal{B}_G\) be such that \(S\) is the set described in  \cref{bigeq}.
    By \cref{lem:techZariski}, we need only show that \(S\neq \varnothing\).

    We will use many partial permutations of \(X\) in building up to an element of \(S\), that is to say bijections between subsets of \(X\).
     We compose and invert these partial permutations as binary relations and denote the composition by concatenation.
     Let \(p_0,p_1,\ldots, p_{m-1}\in X\) be fixed and distinct.
    For a partial permutation \(\sigma\) of \(X\), we define \(d_\sigma:\{0,1,\ldots,m-1\}\to \N\) by
     defining \((i)d_{\sigma}\) to be the largest value of \(j\leq l_i\) such that \(p_i\in \operatorname{dom}(\sigma^{k_{i,0}}a_{i,1}\ldots \sigma^{k_{i,j-1}}a_{i, j})\) (when \(j=0\), we interpret \(\sigma^{k_{i,0}}a_{i,1}\ldots \sigma^{k_{i,j-1}}a_{i, j}\) as the identity with domain \(X\)). 
     For each partial permutation \(\sigma\) and \(i<m\), let 
     \[f_{i,\sigma}:=\sigma^{k_{i,0}}a_{i,1}\ldots \sigma^{k_{i,(i)d_\sigma-1}}a_{i, (i)d_\sigma}.\]
     In particular when \((i)d_{\sigma}=0\), the map \(f_{i,\sigma}\) is the identity function on \(X\).
     We also define
     \[B_\sigma :=\dom(\sigma)\cup \im(\sigma)\cup \{p_0,p_1,\ldots,p_{m-1}\}\cup \{(p_0)f_{0,\sigma},(p_1)f_{1,\sigma},\ldots, (p_{m-1})f_{m-1,\sigma}\}.\]
     Note that \(B_\sigma\) is always finite when \(\sigma\) is finite.
     We inductively build a finite partial permutation \(\sigma\) of \(X\) such that at each stage of the induction, the following conditions hold:
     \begin{enumerate}
      \item[($\sigma$1)] if \(i\neq i'<m\), then \((p_i)f_{i,\sigma}\neq (p_{i'})f_{i',\sigma}\),
         \item[($\sigma$2)] if \(i<m\) and \((i)d_\sigma \neq l_i\), then  \((p_i)f_{i,\sigma}\notin \dom(\sigma)\cup \im(\sigma)\),
         \item[($\sigma$3)] if \(i<m\) and \((i)d_\sigma\neq 0\), then  \((p_i)f_{i,\sigma}\neq p_i\).
     \end{enumerate}
     Moreover, at each stage of the induction we will increase one of the image points of \(d_\sigma\) (which are bounded by the \(l_i\)) so that the induction must eventually terminate. 
     At the end of the induction, we will have \((i)d_\sigma=l_i\) for all \(i<m\). 
     Once we have done this, we can extend \(\sigma\) to an element of \(x\in G\) due to \(G\) being highly transitive. Moreover, from the last condition of the induction, we will have \(p_i\neq (p_i)f_{i,\sigma}=(p_i)f_{i,x}\) for all \(i<m\). In particular we will have \(f_{i,x}\neq 1_G\) for all \(i<m\), so \(x\in S\) as required.

     It remains only to inductively build \(\sigma\).
     Initialize \(\sigma\) as the empty function. In particular \(d_\sigma\) is the constant zero function, each \(f_{i,\sigma}\) is the identity function and \(B_{\sigma}=\{p_0,p_1,\ldots, p_{m-1}\}\).
     
     For the inductive step, choose \(i<m\) such that \((i)d_{\sigma}<l_i\). We keep this \(i\) fixed for the remainder of the proof. The next claim helps us extend \(\sigma\).
     \begin{claim}
         There is a finite set \(C_\sigma\subseteq X\) and a finite partial bijection \(\gamma\) of \(X\) satisfying the following conditions:
         \begin{itemize}
             \item \(B_\sigma\subseteq C_\sigma\), 
             \item \((\dom(\gamma)\cup \im(\gamma))\cap B_{\sigma}=\{(p_i)f_{i,\sigma}\}\),
             \item \((p_i)f_{i,\sigma}\in \dom(\gamma^{k_{i,(i)d_{\sigma}}})\), 
             \item \((\dom(\gamma)\cup \im(\gamma))\backslash C_\sigma=\{(p_i)f_{i,\sigma}\gamma^{k_{i,(i)d_{\sigma}}}\}\), 
             \item \(\dom(\gamma)\cap \dom(\sigma)=\varnothing\),
             \item  \(\im(\gamma)\cap \im(\sigma)=\varnothing\),
             \item \((p_i)f_{i,\sigma}\gamma^{k_{i,(i)d_{\sigma}}}a_{i, (i)d_{\sigma}+1}\notin C_\sigma\)
             \item if \((i)d_{\sigma}+1\neq l_i\), then \((p_i)f_{i,\sigma}\gamma^{k_{i,(i)d_{\sigma}}}a_{i, (i)d_{\sigma}+1}\neq  (p_i)f_{i,\sigma}\gamma^{k_{i,(i)d_{\sigma}}}\).
         \end{itemize} 
     \end{claim}
     \begin{proof}[Proof of Claim]
     We define \(C_\sigma\) to be the set \(\{c_0,c_1,\ldots, c_{|k_{i,(i)d_{\sigma}}|-1}\}\cup B_{\sigma}\) where \(c_0=(p_i)f_{i,\sigma}\) and the points \(c_1,c_2,\ldots,c_{|k_{i,(i)d_{\sigma}}|-1}\in X\backslash B_\sigma\) are arbitrary but distinct.
     We also define another point \(c_{|k_{i,(i)d_{\sigma}}|}\) depending on the following case analysis:
     \begin{enumerate}
         \item[(i)] If \((i)d_\sigma\neq l_i-1\): 
         Choose \(c_{|k_{i,(i)d_{\sigma}}|}\in X\backslash C_\sigma\) such that \((c_{|k_{i,(i)d_{\sigma}}|})a_{i,(i)d_{\sigma}+1}\in X\backslash C_\sigma\) and  \((c_{|k_{i,(i)d_{\sigma}}|})a_{i,(i)d_{\sigma}+1}\neq c_{|k_{i,(i)d_{\sigma}}|}\). This is possible as \(a_{i,(i)d_{\sigma}+1}\) has infinite support and \(C_\sigma\) is finite.
       
       \item[(ii)] If \((i)d_\sigma= l_i-1\): 
         Choose \(c_{|k_{i,(i)d_{\sigma}}|}\in X\backslash C_\sigma\) such that \((c_{|k_{i,(i)d_{\sigma}}|})a_{i,l_i}\in X\backslash C_\sigma\). In this case it is possible that \(a_{i,(i)d_{\sigma}+1}\) is the identity function so we cannot make the same assumption as in case (i).
     \end{enumerate}
     If \(k_{i,(i)d_{\sigma}}>0\), we define \(\gamma: \{c_0,\ldots, c_{|k_{i,(i)d_{\sigma}}|-1}\}\to X\) by \((c_j)\gamma=c_{j+1}\). If \(k_{i,(i)d_{\sigma}}<0\), we define  \(\gamma: \{c_1,\ldots, c_{|k_{i,(i)d_{\sigma}}|}\}\to X\) by \((c_j)\gamma=c_{j-1}\).   

     We need to check that the choice of \(C_\sigma\) and \(\gamma\) satisfy the required conditions. \(B_\sigma \subseteq C_\sigma\) by definition. Note that \(\dom(\gamma)\cup \im(\gamma)=\{c_0, c_1,\ldots, c_{|k_{i,(i)d_{\sigma}}|}\}\) in both cases for \(\gamma\). The only element of this set not explicitly chosen to not belong to \(B_\sigma\) is \(c_0\). We also have \(c_0=(p_i)f_{i,\sigma}\in B_\sigma\). Thus \((\dom(\gamma)\cup \im(\gamma))\cap B_{\sigma}=\{(p_i)f_{i,\sigma}\}\). Also, by ($\sigma$2), it follows that \(\dom(\gamma)\cap \dom(\sigma)=\varnothing\) and \(\im(\gamma)\cap \im(\sigma)=\varnothing\).
     By the choice of \(\gamma\), \(\gamma^{k_{i,(i)d_{\sigma}}}\) has domain \(\{c_0\}=\{(p_i)f_{i,\sigma}\}\) and \((c_0)\gamma^{k_{i,(i)d_{\sigma}}}=c_{|k_{i,(i)d_{\sigma}}|}\).
     The only element of \(\dom(\gamma)\cup \im(\gamma)=\{c_0, c_1,\ldots, c_{|k_{i,(i)d_{\sigma}}|}\}\) which does not belong to \(C_\sigma\) is \(c_{|k_{i,(i)d_{\sigma}}|}=(c_0)\gamma^{k_{i,(i)d_{\sigma}}}=(p_i)f_{i,\sigma}\gamma^{k_{i,(i)d_{\sigma}}}\).
     In both case (i) and case (ii), the element \((c_{|k_{i,(i)d_{\sigma}}|})a_{i, (i)d_{\sigma}+1}=(p_i)f_{i,\sigma}\gamma^{k_{i,(i)d_{\sigma}}}a_{i, (i)d_{\sigma}+1}\) was chosen not to belong to \(C_\sigma\). Finally, if \((i)d_{\sigma}+1\neq l_i\), then from case (i) we also have \((p_i)f_{i,\sigma}\gamma^{k_{i,(i)d_{\sigma}}}=c_{|k_{i,(i)d_{\sigma}}|}\neq (c_{|k_{i,(i)d_{\sigma}}|})a_{i, (i)d_{\sigma}+1}=(p_i)f_{i,\sigma}\gamma^{k_{i,(i)d_{\sigma}}}a_{i, (i)d_{\sigma}+1}\) as required.
     \end{proof}
     
     Let \(C_\sigma\) and \(\gamma\) be as in the claim.
     By the claim, the domains and images of \(\sigma\) and \(\gamma\) are disjoint. Thus we can define a new finite partial permutation \(\beta\) of \(X\) with domain \(\dom(\sigma)\cup \dom(\gamma)\) by
     \[(p)\beta = \begin{cases}
         (p)\sigma & \text{if } p\in\dom(\sigma)\\
         (p)\gamma & \text{if } p\in\dom(\gamma).
     \end{cases}\]
     We need only show that \(\beta\) can serve as \(\sigma\) for the next step of the induction. That is to say, we need only check that the following conditions hold:
\begin{enumerate}
\item[($\beta$0)]  \((i)d_{\beta}>(i)d_{\sigma}\),
      \item[($\beta$1)] if \(j\neq j'<m\), then \((p_j)f_{j,\beta}\neq (p_{j'})f_{j',\beta}\),
         \item[($\beta$2)] if \(j<m\) and \((j)d_\beta \neq l_j\), then  \((p_j)f_{j,\beta}\notin \dom(\beta)\cup \im(\beta)\),
         \item[($\beta$3)] if \(j<m\) and \((j)d_\beta\neq 0\), then  \((p_j)f_{j,\beta}\neq p_j\).
     \end{enumerate}
     
     By construction, the point
     \((p_i)f_{i,\sigma}\beta^{k_{i,(i)d_{\sigma}}}=(p_i)f_{i,\sigma}\gamma^{k_{i,(i)d_{\sigma}}}\) is the unique element of the set \((\dom(\beta)\cup \im(\beta)) \backslash C_\sigma\). In particular \((i)d_\beta\geq (i)d_\sigma+1\). If \((i)d_\sigma=l_i-1\), then we must have \((i)d_\beta=(i)d_\sigma+1\). On the other hand, if \((i)d_\sigma\neq l_i-1\), then the claim shows that \[(p_i)f_{i,\sigma}\beta^{k_{i,(i)d_{\sigma}}}a_{i,(i)d_\sigma+1}=(p_i)f_{i,\sigma}\gamma^{k_{i,(i)d_{\sigma}}}a_{i,(i)d_\sigma+1}\] does not belong to \(C_\sigma\) and is not equal to \((p_i)f_{i,\sigma}\beta^{k_{i,(i)d_{\sigma}}}\). As \((p_i)f_{i,\sigma}\beta^{k_{i,(i)d_{\sigma}}}\) is the only element of \((\dom(\beta)\cup \im(\beta)) \backslash C_\sigma\), the point \((p_i)f_{i,\sigma}\beta^{k_{i,(i)d_{\sigma}}}a_{i,(i)d_\sigma+1}\) belongs to neither the domain nor the image of \(\beta\). 
     So 
     \begin{equation}\tag{$*$}\label{eq:d+1}
         (i)d_\beta=(i)d_\sigma +1
        \end{equation}in all cases. We have now established ($\beta$0). It remains only to check conditions ($\beta$1), ($\beta$2) and ($\beta$3).
     We have also just established that if  \((i)d_\beta \neq l_i\), then 
     \begin{equation}\tag{$**$}\label{eq:newendnowhere}
         (p_i)f_{i,\beta}=(p_i)f_{i,\sigma}\beta^{k_{i,(i)d_{\sigma}}}a_{i,(i)d_\sigma+1}\notin \dom(\beta)\cup \im(\beta)\cup C_{\sigma}.
     \end{equation}

     By the claim, the only element of \((\dom(\gamma)\cup \im(\gamma))\cap B_\sigma \) is the point \((p_i)f_{i,\sigma}\).
     By ($\sigma$1) and ($\sigma$2), we have \((p_{i'})f_{i',\sigma}\in B_\sigma\backslash (\dom(\sigma)\cup \im(\sigma))\) and \((p_i)f_{i,\sigma}\neq (p_{i'})f_{i',\sigma}\) for \(i'\neq i\).
     It follows that \((p_{i'})f_{i',\sigma}\not\in \dom(\beta)\cup \im(\beta)\) when \(i'\neq i\). Thus when \(i'\neq i\), we have 
     \begin{equation}\tag{$***$}\label{eq:otheris}
     (i')d_\sigma =(i')d_\beta\text{ and }(p_{i'})f_{i',\beta}=(p_{i'})f_{i',\sigma}\in B_{\sigma}.
     \end{equation}
     From \cref{eq:d+1} and \cref{eq:otheris}, we have now entirely determined the map \(d_{\beta}\) and the partial map \(f_{i,\beta}\).
     From the claim we have \((\dom(\gamma)\cup \im(\gamma))\cap B_{\sigma}=\{(p_i)f_{i,\sigma}\}\). This together with \cref{eq:otheris} and ($\sigma$2) show that \((p_{i'})f_{i',\beta}=(p_{i'})f_{i',\sigma}\notin \dom(\beta)\cup \im(\beta)\) for \(i\neq i'\) with \((i')d_\sigma\neq l_{i'}\).
     By \cref{eq:newendnowhere}, we also have (\(p_{i})f_{i,\beta}\notin \dom(\beta)\cup \im(\beta)\) when \((i')d_\sigma\neq l_{i'}\). We have now shown that condition ($\beta$2) holds. 
     Only conditions ($\beta$1) and ($\beta$3) remain.
     
      By \cref{eq:d+1} and the claim, we have \((p_i)f_{i,\beta}\notin C_\sigma\supseteq B_\sigma\). Thus, by \cref{eq:otheris}, \((p_i)f_{i,\beta}\) is distinct from each of \((p_{i'})f_{i',\beta}=(p_{i'})f_{i',\sigma}\) with \(i'\neq i\). The points \((p_{i'})f_{i',\beta}=(p_{i'})f_{i',\sigma}\) for various \(i'\neq i\) are also distinct by \cref{eq:otheris} and ($\sigma$1). Together these establish that condition ($\beta$1) holds. Only condition ($\beta$3) remains.

     Suppose now that \(i'<m\) and \((i')d_\beta \neq 0\).
     There are two cases. If \(i'\neq i\), then by \cref{eq:otheris} we have \((i')d_\sigma=(i')d_\beta \neq 0\) and \((p_{i'})f_{i',\beta}=(p_{i'})f_{i',\sigma}\).
     By ($\sigma$3), it follows that  \((p_{i'})f_{i',\beta}\neq p_{i'}\) as required.
     The remaining case is when \(i'=i\). By \cref{eq:d+1} and the claim,
     the point \((p_i)f_{i,\beta}\) does not belong to \(C_\sigma\).
     In particular, \((p_i)f_{i,\beta}\neq p_i\in B_\sigma\subseteq C_\sigma\).
     
\end{proof}

We can now finish the proof of \cref{maingeneral}. 
Moreover \cref{mainV} is shown as a special case.

\begin{theorem}\label{topcor}
    Suppose that \(X\) is a Hausdorff topological space with no isolated points and \(G\) acts highly transitively on \(X\) by homeomorphisms.
    In this case the group Zariski topology on \(G\) is irreducible.
\end{theorem}
\begin{proof}
    To apply \cref{main}, we need only show that every non-trivial element has infinite support. 
    This follows as the support of any homeomorphism of a Hausdorff space is open, and a Hausdorff space with no isolated points has no non-empty finite open subsets.
\end{proof}

\begin{theorem}
    The group Zariski topologies on Thompson's group  \(V\) (or any of the Thompson groups \(V_n\) with \(n\geq 2\)) and the group \(\operatorname{Homeo}(2^\omega)\) are irreducible.
\end{theorem}
\begin{proof}
As the group \(\operatorname{Homeo}(2^\omega)\) acts highly transitively on the Cantor space, we can directly apply \cref{topcor} to show that the group Zariski topology on \(\operatorname{Homeo}(2^\omega)\) is irreducible. 

The group \(V_n\) does not act transitively on the Cantor space \(\{0,1,\ldots, n-1\}^\omega\) (the group is countable) so we restrict our attention to a single orbit. 
It is well-known however (see for example \cite{hightrans}) that the action of \(V_n\) on any of its orbits is highly transitive. It thus suffices to show that any of its orbits has no isolated points as a subspace of \(\{0,1,\ldots, n-1\}^\omega\). 

Let \(x=x_0x_1x_2\ldots \in \{0,1,\ldots, n-1\}^\omega\) be arbitrary but fixed.  We must show that \(x\) is not isolated in the subspace \((x)V_n\) of \(\{0,1,\ldots, n-1\}^\omega\). For each \(m\in \N\) fix \(y_m\in \{0,1,\ldots,n-1\}\backslash \{x_m\}\). For each \(m\) we can define an element \(f_m\in V_n\) by swapping the prefixes \(x_0x_1\ldots x_{m-1}x_m\) and \(x_0x_1\ldots x_{m-1}y_m\) and fixing all other points. 
It follows that the elements \(x_0\ldots x_{m-1}y_mx_{m+1}\ldots\) belong to the orbit of \(x\) for all \(m\in \N\). 
By definition, these elements are all distinct from \(x\) and converge to \(x\). The result follows.
\end{proof}

\section{The Interval and the Circle}
In this section, the main tool is \cref{mainFTlemma} so we start by establishing the conventions needed there.
For convenience, we view the topological space \(S^1\) as the set \([0,1)\) with the topology induced by the metric 
\[d(x,y)=\operatorname{min}(|x-y|,1-|x-y|).\]

Under the compact-open topology, the topological group $\operatorname{Homeo}([0,1])$ is topologically isomorphic to the stabilizer of $0$ in the topological group $\operatorname{Homeo}(S^1)$.
As such, we will for simplicity work entirely in the group \(\operatorname{Homeo}(S^1)\).
Given \(x\neq y\in S^1=[0,1)\), we write \((x,y)\) for the \emph{open interval from \(x\) to \(y\)} using the cyclic order on \(S^1\). That is to say
\[(x,y):=\makeset{z\in [0,1)}{\(x<z<y\) or \(y<x<z\) or \(z<y<x\)}.\]
We refer to an arbitrary set of this form as simply an \emph{open interval}. Note that we consider neither the empty set nor the entire space to be an open interval.
As a subspace of \(S^1\), each open interval is open, connected and homeomorphic to \(\mathbb{R}\).
It is well-known that every homeomorphism of \(S^1\) either preserves or reverses the cyclic order. In particular the following proposition will be useful to us.
\begin{prop}\label{interval-preserve}
    If \(h\in \operatorname{Homeo}(S^1)\) and \(I\) is an open interval of \(S^1\), then \((I)h\) is also an open interval of \(S^1\).
\end{prop}

In our arguments we will need to keep track of the size of these intervals and their unions.
    We denote by \(\lambda\) the usual Lebesgue measure on \([0,1)=S^1\). For example, \(\lambda(\varnothing)=0\), \(\lambda(S^1)=1\),  \(\lambda(\frac{1}{5}, \frac{3}{5})=\frac{2}{5}\) and \(\lambda(\frac{3}{5}, \frac{1}{5})=\frac{3}{5}\). A key observation will be that if any interval \(I\) contains points of distance at least \(\varepsilon\), then the measure of the interval is at least \(\varepsilon\).

\begin{lemma}\label{mainFTlemma}
    Suppose that \(G\leq \Homeo(S^1)\) and there is a dense subset \(D\) of \([0, 1)\) such that for all \(p,q\in D\) with \(p<q\), there is \(t_{p,q}\in G\) with support \((p,q)\).
    In this case \(\mathfrak{Z}_G\) is the compact-open topology.
\end{lemma}
\begin{proof}
Let \(\Tau_{co}\) be the compact-open topology on \(G\).
By \cref{compact-open-good}, \(\Tau_{co}\) is a group topology induced by the uniform metric. In particular, as a Hausdorff group topology, \(\Tau_{co}\supseteq \mathfrak{Z}_G\). We need only show the reverse containment.
As shifting by elements of \(G\) is a homeomorphism with respect to \(\mathfrak{Z}_G\), we only need to show the each neighbourhood of \(1_G\) with respect to \(\Tau_{co}\) is a \(\mathfrak{Z}_G\)-neighbourhood of \(1_G\). 
The \(\Tau_{co}\)-neighbourhoods of \(1_G\) have a basis consisting of the sets
\[B_{d_\infty}(1_G,\varepsilon):= \makeset{f\in G}{for all \(x\in S^1\), \(d(x, (x)f)<\varepsilon\)}\]
where \(\varepsilon >0\).

Let \(\varepsilon >0\). We may assume without loss of generality that \(\varepsilon<1\).
As \(D\) is dense, we can find \(p_{1},\ldots, p_{k-1}\in D\backslash\{0\}\) such that for all \(i<k\), we have \(0<p_{i+1}-p_i<\frac{\varepsilon}{32}\), \(0< p_1<\frac{\varepsilon}{32}\) and \(1-\frac{\varepsilon}{32}<p_{k-1}<1\).
Define \(p_0=0\) and \(p_k=1\).
For each \(0<i<k\), also let \(p_{i,l},p_{i,r}\in D\) be such that for all \(0\leq j <k\) we have \(0<p_{i,r}-p_{i,l}<\frac{1}{2k}\min(p_{i+1}-p_i,p_{i}-p_{i-1})<\frac{\varepsilon}{64k}\) and \(p_{i,l}<p_i<p_{i,r}\). We can see the layout of the points \(p_i\), \(p_{i,l}\) and \(p_{i,r}\) in Figure~\ref{mainFTlemma}. 

\begin{figure}
\begin{center}
\vspace{-2mm}
\includegraphics[scale=0.3]{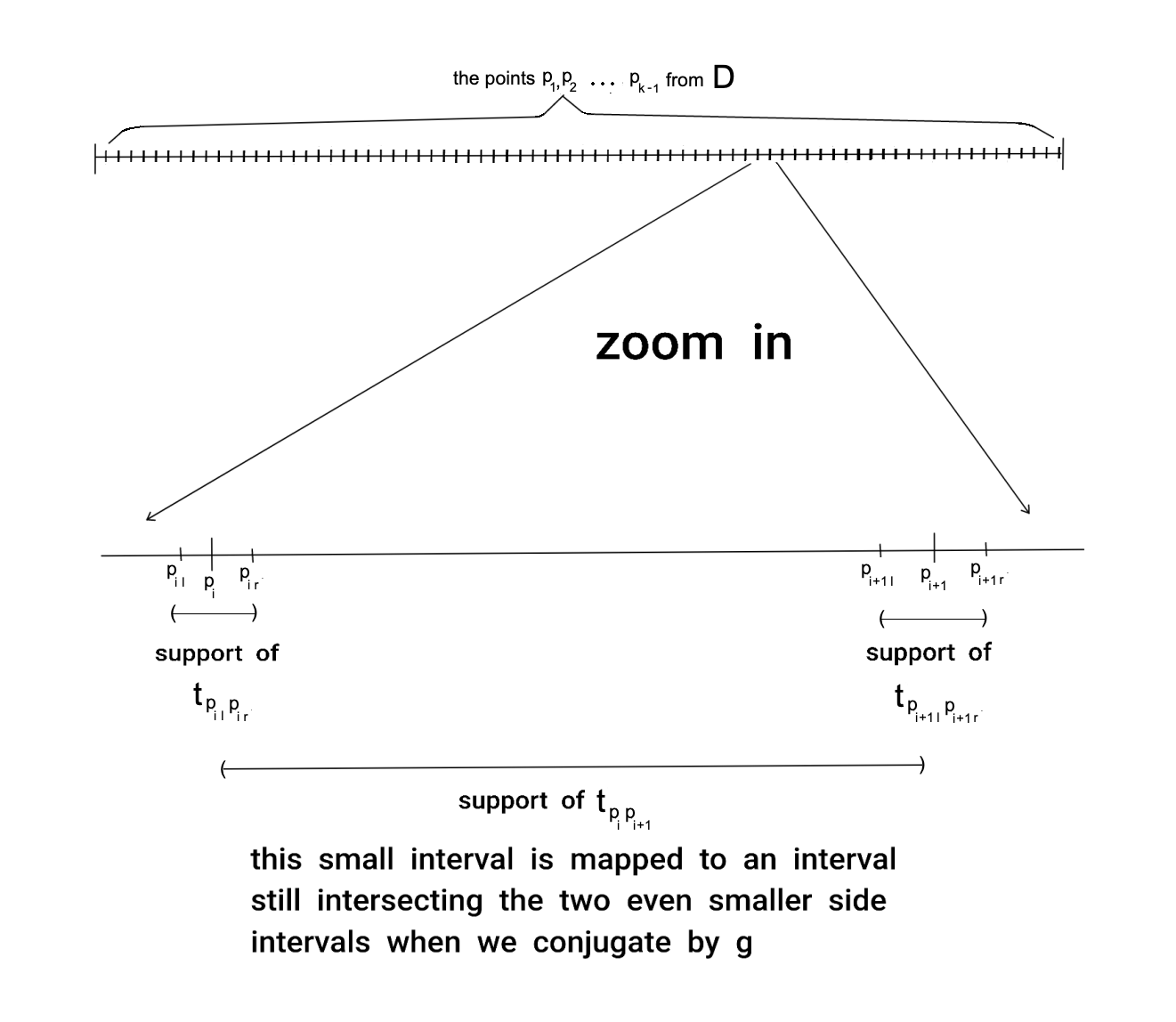}
\end{center}
 \caption{Figure of the distribution of the points $p_{i}, p_{i,l}, p_{i,r}$ in the set up of the proof of \cref{mainFTlemma}}
    \label{mainFTlemma}
\end{figure}
The idea is that we will model each point \(p_i\) with the element \(t_{p_{i,l},p_{p_{i,r}}}\) which is supported on a small neighbourhood of \(p_i\) and model each interval \((p_i, p_{i+1})\) for \(0<i<k-1\) by the element \(t_{p_i, p_{i+1}}\).
We then model moving the interval \((p_i, p_{i+1})\) ``not very much'' by conjugating \(t_{p_i, p_{i+1}}\) to an element whose support touches points near each of \(p_i\) and \(p_{i+1}\).

Accordingly, we define
    \[U_{i,\varepsilon}:=\makeset{g\in G}{\([t_{p_i,p_{i+1}}^{g},t_{p_{i,l},p_{i,r}}]\neq 1_G\) and \([t_{p_i,p_{i+1}}^{g},t_{p_{i+1,l},p_{i+1,r}}]\neq 1_G\)}\]
    for \(0<i<k-1\). The elements of \(U_{i, \varepsilon}\) are defined to act like the element \(g\) from Figure~\ref{mainFTlemma}.
    As the support of \(t_{p_{i,l}, p_{i,r}}\) is \((p_{i,l}, p_{i,r})\), the conjugate of \(t_{p_{i,l}, p_{i,r}}\) by \(t_{p_{i}, p_{i+1}}\) has support which is an interval between
\((p_{i,l})t_{p_{i}, p_{i+1}}\) and \((p_{i,r})t_{p_{i}, p_{i+1}}\). As \(t_{p_{i}, p_{i+1}}\) fixes \(p_{i,l}\) and does not fix \(p_{i,r}\), it follows that \(t_{p_{i}, p_{i+1}}\) and \(t_{p_{i,l}, p_{i,r}}\) do not commute.
Similarly, \(t_{p_{i}, p_{i+1}}\) and \(t_{p_{i+1,l}, p_{i+1,r}}\) do not commute. Thus \(1_G\in U_{i, \varepsilon}\) for all \(0<i<k-1\).
Let \(U_{\varepsilon}:=\cap_{0<i<k-1} U_{i, \varepsilon}\). We need only show that
\[1_G\in U_\varepsilon \subseteq B_{d_\infty}(1_G,\varepsilon).\]
We have already shown that \(1_G\in U_\varepsilon\).
Let \(g\in U_\varepsilon\) and let \(x\in S^1\). We need only show that \(d((x)g,x)<\varepsilon\).
Let \(i<k\) be fixed such that \(x\in [p_i,p_{i+1}]\). 
We wish to leverage the ideas from the diagram to show that none of our small intervals \((p_{i},,p_{i+1})\) can be moved by \(g\) very much. To conclude that \(g\in B_{d_\infty}(1_G,\varepsilon)\), we must also show that the interval \((p_{i},p_{i+1})g\) cannot be too large. 
Informally, we do this by showing that if \((p_{i},p_{i+1})g\) is large then there is not `enough space' for the remaining intervals.
Note that the intervals \((p_0, p_1),(p_1,p_2)\ldots, (p_{k-1},p_k)\) are disjoint and contain all but finitely many points of \(S^1\).
Thus by \cref{interval-preserve}, the sets 
\[(p_0, p_1)g,(p_1,p_2)g\ldots (p_{k-1},p_k)g\]
are disjoint open intervals which contain all but finitely many points of \(S^{1}\). It follows that
\[\sum_{0\leq j\leq k-1} \lambda((p_j, p_{j+1})g)=1.\]
By the choice of \(g\), the element \(t_{p_{j},p_{j+1}}^g\) does not commute with \(t_{p_{j,l},p_{j,r}}\) for \(0<j<k-1\). 
It follows that the support of \(t_{p_{j},p_{j+1}}^g\) intersects the support of \(t_{p_{j,l},p_{j,r}}\) for all \(0<j<k-1\).
Thus for \(0<j<k-1\), the interval \((p_j, p_{j+1})g\) intersects the interval \((p_{j,l},p_{j,r})\).
By the same reasoning, \((p_j, p_{j+1})g\) intersects the interval \((p_{j+1,l},p_{j+1,r})\) for all \(0<j<k-1\). 
By the assumptions $p_{j,r}-p_{j,l}<\frac{1}{2}\min(p_{j+1}-p_j,p_{j}-p_{j-1})$, $\varepsilon<1$ and $p_{j+1}-p_j<\frac{\varepsilon}{32}$, we have \[p_{j-1}<p_{j,l}<p_{j}<p_{j,r}<p_{j+1,l}<p_{j+1}<p_{j+1,r}<p_{j+2}\] and \(|p_{j-1}-p_{j+2}|< \frac{1}{16}\). 
Recall that all of the displayed points other than possibly the first and last belong to \(D\).
 As the interval \((p_j, p_{j+1})g\) intersects both \((p_{j,l}, p_{j,r})\) and \((p_{j+1,l}, p_{j+1,r})\), the set  \((p_j, p_{j+1})g\) contains two points of distance at least \(|p_{j,r}- p_{j+1,l}|\) for all for \(0<j<k\). Thus, recalling that $p_{j,r}-p_{j,l}$ and $p_{j+1,r}-p_{j+1,l}$ are at most $\frac{\varepsilon}{64k}$, we have
\[\lambda((p_j, p_{j+1})g) \geq |p_{j,r}- p_{j+1,l}|\geq p_{j+1}- p_{j}-|p_{j,r}-p_j|-|p_{j+1}-p_{j+1,l}| \geq p_{j+1}-p_j - 2\frac{\varepsilon}{64k}=p_{j+1}-p_j - \frac{\varepsilon}{32k}\]
for \(0<j<k-1\). We have shown now that each interval (except possibly the end intervals) must retain most of its length after application of $g$. This forces the intervals to also stay reasonably small as we see below
\begin{align*}
    \lambda((p_i, p_{i+1})g)&= 1-\sum_{\substack{0\leq j<k\\j\neq i}}\lambda((p_j, p_{j+1})g)\\
    &\leq 1-\sum_{\substack{0< j<k-1\\j\neq i}}\lambda((p_j, p_{j+1})g)\\
    &\leq 1-\sum_{\substack{0< j<k-1\\j\neq i}}(p_{j+1}-p_j-\frac{\varepsilon}{32k})\\
    &\leq \frac{\varepsilon}{32}+1-\sum_{\substack{0< j<k-1\\j\neq i}}(p_{j+1}-p_j)\\
    &\leq (p_{i+1}-p_i)+(p_1-p_0)+(p_k-p_{k-1})+\frac{\varepsilon}{32}+1-\sum_{\substack{0\leq j<k}}(p_{j+1}-p_j)\\
    &\leq \frac{3\varepsilon}{32}+\frac{\varepsilon}{32}< \frac{\varepsilon}{8}+\frac{\varepsilon}{32}< \frac{\varepsilon}{4}.
\end{align*}
The above argument also shows that \(\lambda((p_0, p_{1})g),\lambda((p_1, p_{2})g), \lambda((p_{k-2}, p_{k-1})g),\lambda((p_{k-1}, p_{k})g)< \frac{\varepsilon}{4}\). 
We now know that \(g\) maps our intervals to new intervals of a similar size so we can almost finish the proof.
However, we still need to analyse endpoints as we have not defined points \(p_{0,l}\), \(p_{0,r}\), \(p_{1,l}\) and \(p_{1,r}\). To fix this, we define a value \(i'\in \{1,\ldots, k-1\}\) depending on cases:
\begin{enumerate}
    \item If \(i=0\), let \(i'=1\),
    \item if \(i=k-1\), let \(i'=k-2\),
    \item if \(i\notin\{0, k-1\}\), let \(i'=i\).
\end{enumerate}
In each case we know that \(p_{i'},p_{i'+1}\in D\). Also
\[d(x, p_{i'})\leq d(x,p_i)+d(p_{i}, p_{i'})\leq \frac{\varepsilon}{32}+\frac{\varepsilon}{32}=\frac{\varepsilon}{16}\] and \(\lambda((p_{i'}, p_{i'+1})g) < \frac{\varepsilon}{4}\). Recall that both the shorter interval between \((p_i)g\) and \((p_{i+1})g\), as well as the shorter interval between \((p_i)g\) and \((p_{i'})g\) have length less than \(\frac{\varepsilon}{4}\). 
So we have
\[d((x)g, (p_{i'})g)\leq d((x)g, (p_i)g)+d((p_{i})g, (p_{i'})g)< \frac{\varepsilon}{4}+\frac{\varepsilon}{4}=\frac{\varepsilon}{2}.\]
Moreover, we have shown earlier that there is a point \(b\in (p_{i'}, p_{i'+1})g\cap (p_{i',l},p_{i',r})\).
In conclusion, recalling that \(\lambda((p_{i'}, p_{i'+1})g)<\frac{\varepsilon}{4}\), we have
\begin{align*}
    d(x,(x)g)&\leq d(x,p_{i'})+d(p_{i'},b)+ d(b,(p_{i'})g)+d((p_{i'})g,(x)g) \\
    &\leq \frac{\varepsilon}{16}+\frac{\varepsilon}{32}+ \frac{\varepsilon}{4}+\frac{\varepsilon}{2}<\varepsilon
\end{align*}
as required.

\end{proof}

\cref{mainF} and \cref{mainT} now follow from the Lemma.

\begin{theorem}
    If \(G\leq \Homeo(S^1)\) is any group containing Thompson's group \(T\) (or any of the Thompson groups \(T_n\) with \(n\geq 2\)), then the group Zariski topology on \(G\) is the compact-open topology. In particular, it is a Hausdorff group topology.
\end{theorem}
\begin{proof}
    Taking \(D=\Z[1/n] \cap [0,1)\), the Thompson group \(T_n\) contains the elements required by \cref{mainFTlemma}.
    The compact-open topology is a Hausdorff group topology by \cref{compact-open-good}.
\end{proof}

\begin{theorem}
    If \(G\leq \Homeo([0,1])\) is any group containing Thompson's group \(F\) (or any of the Thompson groups \(F_n\) with \(n\geq 2\)), then the group Zariski topology on \(G\) is the compact-open topology. In particular, it is a Hausdorff group topology.
\end{theorem}
\begin{proof}
    For this proof we identify \(\Homeo([0,1])\) with the stabilizer of the point \(0\) in the group \(\Homeo(S^1)\).
    Note that the compact-open topology is the same whether \(\Homeo([0,1])\) acts on \([0,1]\) or \(S^1\).
    Taking \(D=\Z[1/n] \cap [0,1)\), the Thompson group \(F_n\leq \Homeo([0,1])\leq \operatorname{Homeo}(S^1)\) contains the elements required by \cref{mainFTlemma}.
    The compact-open topology is a Hausdorff group topology by \cref{compact-open-good}.
\end{proof}


\section{Manifolds}

In this section we prove \cref{mainManifold}. Most of the proof for this has been established in previous sections so we can go straight to the main result.

\begin{theorem}
    Suppose that \(n\in \N\) and \(M\) is an \(n\)-manifold with exactly one connected component. Then the following are equivalent:
    \begin{enumerate}
        \item the group Zariski topology on \(\operatorname{Homeo}(M)\) is Hausdorff,
        \item the group Zariski topology on \(\operatorname{Homeo}(M)\) is a group topology,
        \item \(n\leq 1\).
    \end{enumerate}
\end{theorem}
\begin{proof}
    \((2)\Rightarrow (1):\) The group Zariski topology is always a \(T_1\) topology.
    It is well known that every \(T_1\) topological group is Tychonoff. In particular every \(T_1\) topological group is Hausdorff.

\((1)\Rightarrow (3):\) 
By \cref{maingeneral} + \cref{hightrans}, the Zariski topology is not Hausdorff when \(n\geq 2\).

\((3)\Rightarrow (2):\)
By \cref{1manclass}, the only manifolds of dimension at most \(1\) with exactly one connected component are \(\mathbb{R}\), \(S^1\), and a single point space.
In the case of a point, \(\operatorname{Homeo}(M)\) is the trivial group so is discrete.
In the case of \(\mathbb{R}\), we have \(\operatorname{Homeo}(M)\cong \operatorname{Homeo}([0, 1])\) so we apply \cref{mainF}.
The case of \(S^1\) follows from \cref{mainT}.
\end{proof}
\section{Questions}
By \cref{mainV}, the group Zariski topology on \(V\) is not Hausdorff and so \cite{markov1946unconditionally} implies that the Markov topology on \(V\) is not Hausdorff either.
As each orbit of \(V\) is dense, this group acts faithfully on any of its orbits. 
As such it inherits a group topology from the pointwise topology on the group \(\operatorname{Sym}(o)\) where \(o\) is an orbit of \(V\) (the topology on \(\operatorname{Sym}(o)\) treats \(o\) as a discrete set).
It is not clear what the intersection of these topologies is, in particular whether the intersection is already the Zariski-Markov topology on \(V\).
On the other hand \cref{mainF} and \cref{mainT} show that every Hausdorff group topology on \(F\) and \(T\) contains the compact-open topology.
They also admit the pointwise topology inherited from the symmetric groups on any of their orbits as well as the discrete topology. This leads us to the following questions.
\begin{question}
    Are there any other Hausdorff group topologies on any of Thompson groups \(F\), \(T\) and \(V\) other than the pointwise topologies (with respect to one of the orbits of its usual action), the compact-open topology, joins of these, and the discrete topology.
\end{question}
\begin{question}
    How many Hausdorff group topologies do the groups \(F\), \(T\) and \(V\) admit?
\end{question}
The results of \cite{markov1946unconditionally, DDSabelianzariski} imply that the Zariski and Markov topologies agree for countable groups and abelian groups. 
The group \(\operatorname{Homeo}(2^\omega)\) is in neither of these classes. As such, given \cref{mainV}, the following are natural questions.
\begin{question}
Is the Markov topology on \(\operatorname{Homeo}(2^\omega)\) equal to the Zariski topology?
\end{question}
\begin{question}
      Is the Markov topology on \(\operatorname{Homeo}(2^\omega)\) Hausdorff?
\end{question}

If we give \(2^\omega\) the lexicographic order and the induced circular order from this total order, then one can show that the subgroup of \(\operatorname{Homeo}(2^\omega)\) which preserves or reverses this circular order is isomorphic to \(\operatorname{Homeo}(S^1)\).
As such, \cref{mainV} in a sense says that all sufficiently large subgroups of \(\operatorname{Homeo}(2^\omega)\) which preserve the circular order have a Hausdorff Zariski topology.
Conversely, \cref{maingeneral} says that all highly transitive subgroups of \(\operatorname{Homeo}(2^\omega)\) do not have a Hausdorff Zariski topology. This leads us to the following question.

\begin{question}
    Can one classify which ``sufficiently large" subgroups of \(\operatorname{Homeo}(2^\omega)\) have a Hausdorff Zariski topology (for some reasonable notion of sufficiently large).
\end{question}

\section*{Acknowledgements}
 The author would like to thank the anonymous referee for their helpful comments and
careful reading of the paper. The author would also like to thank Naftoli Kolodny for comments on the previous draft.
\bibliography{bib}

@article{bryant1977verbal,
  title={The verbal topology of a group},
  author={Bryant, R},
  journal={Journal of Algebra},
  volume={48},
  number={2},
  pages={340--346},
  year={1977},
  publisher={Academic Press}
}

@article{hightrans,
 author = {Le Boudec, Adrien and Matte Bon, Nicol{\'a}s},
 title = {Confined subgroups and high transitivity},
 fjournal = {Annales Henri Lebesgue},
 journal = {Ann. Henri Lebesgue},
 issn = {2644-9463},
 volume = {5},
 pages = {491--522},
 year = {2022},
 language = {English},
 doi = {10.5802/ahl.128},
 zbMATH = {7524691},
 Zbl = {1553.20009}
}

@article{markov1946unconditionally,
  title={On unconditionally closed sets},
  author={Markov, AA},
  journal={Mat. Sbornik},
  volume={18},
  number={1},
  pages={3--28},
  year={1946}
}

@article{DDSabelianzariski,
author = {Dikranjan, D and Shakhmatov, D},
year = {2007},
month = {04},
pages = {},
title = {Reflection principle characterizing groups in which unconditionally closed sets are algebraic},
volume = {11},
journal = {Journal of Group Theory},
doi = {10.1515/JGT.2008.025}
}

@article{banakh2012algebraically,
  title={Algebraically determined topologies on permutation groups},
  author={Banakh, T and Guran, I and Protasov, I},
  journal={Topology and its Applications},
  volume={159},
  number={9},
  pages={2258--2268},
  year={2012},
  publisher={Elsevier}
}

@article{bardyla2025note,
  title={A note on intrinsic topologies of groups},
  author={Bardyla, S and Elliott, L and Mitchell, J and P{\'e}resse, Y},
  journal={arXiv preprint arXiv:2506.11500},
  year={2025}
}

@phdthesis{elliott2022constructing,
  title={On constructing topology from algebra},
  author={Elliott, L},
  year={2022},
  school={University of St Andrews, arXiv:2601.13279}
}

@book{kechris1995classical,
  title={Classical Descriptive Set Theory},
  author={Kechris, A},
  series={Graduate Texts in Mathematics},
  volume={156},
  year={1995},
  publisher={Springer-Verlag},
  address={New York},
  isbn={978-0-387-94374-9}
}

@article{zbMATH05232892,
 author = {Sipacheva, O},
 title = {Unconditionally {{\({{\tau}} \)}}-closed and {{\({{\tau}} \)}}-algebraic sets in groups},
 fjournal = {Topology and its Applications},
 journal = {Topology Appl.},
 issn = {0166-8641},
 volume = {155},
 number = {4},
 pages = {335--341},
 year = {2008},
 language = {English},
 doi = {10.1016/j.topol.2007.04.013},
 zbMATH = {5232892},
 Zbl = {1156.54021}
}

@incollection{zbMATH06300237,
 author = {Dikranjan, D. and Toller, D.},
 title = {Markov's problems through the looking glass of {Zariski} and {Markov} topologies},
 booktitle = {Ischia group theory 2010. Proceedings of the conference in group theory, Ischia, Naples, Italy, April 14--17, 2010.},
 isbn = {978-981-4350-38-9; 978-981-4350-05-1},
 pages = {87--130},
 year = {2012},
 publisher = {Hackensack, NJ: World Scientific},
 language = {English},
 zbMATH = {6300237},
 Zbl = {1295.22002}
}

@article{zbMATH07781608,
 author = {Bonatto, M and Dikranjan, D and Toller, D},
 title = {Groups with cofinite {Zariski} topology and potential density},
 fjournal = {Topology and its Applications},
 journal = {Topology Appl.},
 issn = {0166-8641},
 volume = {340},
 pages = {29},
 note = {Id/No 108720},
 year = {2023},
 language = {English},
 doi = {10.1016/j.topol.2023.108720},
 zbMATH = {7781608},
 Zbl = {1535.22001}
}

@article{zbMATH06893997,
 author = {Dikranjan, D and Toller, D},
 title = {Zariski topology and {Markov} topology on groups},
 fjournal = {Topology and its Applications},
 journal = {Topology Appl.},
 issn = {0166-8641},
 volume = {241},
 pages = {115--144},
 year = {2018},
 language = {English},
 doi = {10.1016/j.topol.2018.03.025},
 zbMATH = {6893997},
 Zbl = {1478.22002}
}

@article{interval1,
title = {Order and minimality of some topological groups},
journal = {Topology and its Applications},
volume = {201},
pages = {131-144},
year = {2016},
note = {2014 International Conference on Topology and its Applications, Nafpaktos, Greece},
issn = {0166-8641},
doi = {https://doi.org/10.1016/j.topol.2015.12.032},
url = {https://www.sciencedirect.com/science/article/pii/S0166864115005635},
author = {Michael Megrelishvili and Luie Polev}
}

@book{munkres2000topology,
  title={Topology},
  author={Munkres, J},
  isbn={9780131816299},
  series={Featured Titles for Topology},
  url={https://books.google.com/books?id=XjoZAQAAIAAJ},
  year={2000},
  publisher={Prentice Hall, Incorporated}
}

@article{marimon2025guide,
  title={A guide to topological reconstruction on endomorphism monoids and polymorphism clones},
  author={Marimon, P and Pinsker, M},
  journal={arXiv preprint arXiv:2512.01086},
  year={2025}
}

@article{pinsker2023zariski,
  title={On the {Zariski} topology on endomorphism monoids of omega-categorical structures},
  author={Pinsker, M and Schindler, C},
  journal={The Journal of Symbolic Logic},
  pages={1--19},
  year={2023},
  publisher={Cambridge University Press}
}

@article{chang2017minimum,
  title={Minimum topological group topologies},
  author={Chang, Xiao and Gartside, Paul},
  journal={Journal of Pure and Applied Algebra},
  volume={221},
  number={8},
  pages={2010--2024},
  year={2017},
  publisher={Elsevier}
}

@article{GARTSIDE2003103,
title = {Autohomeomorphism groups},
journal = {Topology and its Applications},
volume = {129},
number = {2},
pages = {103-110},
year = {2003},
issn = {0166-8641},
doi = {https://doi.org/10.1016/S0166-8641(02)00140-2},
url = {https://www.sciencedirect.com/science/article/pii/S0166864102001402},
author = {Paul Gartside and Aneirin Glyn}
}

@article{dikranjan2012markov,
  title={The {Markov} and {Zariski} topologies of some linear groups},
  author={Dikranjan, D and Toller, D},
  journal={Topology and its Applications},
  volume={159},
  number={13},
  pages={2951--2972},
  year={2012},
  publisher={Elsevier}
}

@article{banakh2010zariski,
  title={{Zariski} topologies on groups},
  author={Banakh, T and Protasov, I},
  journal={arXiv preprint arXiv:1001.0601},
  year={2010}
}

@article{goffer2024note,
  title={A note on the {Zariski} topology on groups and semigroups},
  author={Goffer, G and Greenfeld, B},
  journal={Journal of Algebra},
  volume={651},
  pages={111--118},
  year={2024},
  publisher={Elsevier}
}

@article{elliott2023automatic,
  title={Automatic continuity, unique {Polish} topologies, and {Zariski} topologies on monoids and clones},
  author={Elliott, L and Jonu{\v{s}}as, J and Mesyan, Z and Mitchell, J and Morayne, M and P{\'e}resse, Y},
  journal={Transactions of the American Mathematical Society},
  volume={376},
  number={11},
  pages={8023--8093},
  year={2023}
}

@article{cannon1996introductory,
  title={Introductory notes on {Richard} {Thompson’s} groups},
  author={Cannon, J W},
  journal={Enseign. Math.(2)},
  volume={42},
  number={3-4},
  pages={215},
  year={1996}
}

@article{arens1946topologies,
  title={Topologies for homeomorphism groups},
  author={Arens, R},
  journal={American Journal of Mathematics},
  volume={68},
  number={4},
  pages={593--610},
  year={1946},
  publisher={JSTOR}
}

@article{mann2021structure,
  title={The structure of homeomorphism and diffeomorphism groups},
  author={Mann, K},
  journal={Notices of the American Mathematical Society},
  volume={68},
  number={4},
  pages={482--493},
  year={2021},
  url={https://www.ams.org/notices/202104/rnoti-p482.pdf}
}

@article{gale1987classification,
  title={The classification of 1-manifolds: a take-home exam},
  author={Gale, D},
  journal={The American Mathematical Monthly},
  volume={94},
  number={2},
  pages={170--175},
  year={1987},
  publisher={Taylor \& Francis}
}

@article{mostert1961reasonable,
  title={Reasonable topologies for homeomorphism groups},
  author={Mostert, P},
  journal={Proceedings of the American Mathematical Society},
  volume={12},
  number={4},
  pages={598--602},
  year={1961},
  publisher={JSTOR}
}

@article{dijkstra2005homeomorphism,
  title={On homeomorphism groups and the compact-open topology},
  author={Dijkstra,  J},
  journal={The American Mathematical Monthly},
  volume={112},
  number={10},
  pages={910--912},
  year={2005},
  publisher={Taylor \& Francis}
}

@article{dieudonne1948topological,
  title={On topological groups of homeomorphisms},
  author={Dieudonn{\'e}, J},
  journal={American Journal of Mathematics},
  volume={70},
  number={3},
  pages={659--680},
  year={1948},
  publisher={JSTOR}
}

@article{belk2025short,
  title={A short proof of {Rubin’s} theorem},
  author={Belk, J and Elliott, L and Matucci, F},
  journal={Israel Journal of Mathematics},
  volume={267},
  number={1},
  pages={157--169},
  year={2025},
  publisher={Springer}
}

@article{rubin1989reconstruction,
  title={On the reconstruction of topological spaces from their groups of homeomorphisms},
  author={Rubin, M},
  journal={Transactions of the American Mathematical Society},
  volume={312},
  number={2},
  pages={487--538},
  year={1989}
}
\bibliographystyle{alpha}
\end{document}